\documentclass[12pt]{amsart}
\usepackage{latexsym,amssymb,amscd}

\def\NZQ{\mathbb}               % the font for N,Z,Q,R,C

\def\ZZ{{\NZQ Z}}

\def\B'c{{\mathcal{B'}}}
\def\U'c{{\mathcal{U'}}}

\def\opn#1#2{\def#1{\operatorname{#2}}} % to make operators
\opn\chara{char}
\opn\length{\ell}
\opn\projdim{proj\,dim}
\opn\injdim{inj\,dim}
\opn\ini{in}
\opn\rank{rank}
\opn\depth{depth}
\opn\sdepth{sdepth}
\opn\height{ht}
\opn\embdim{emb\,dim}
\opn\codim{codim}

\opn\Tr{Tr}
\opn\bigrank{big\,rank}
\opn\superheight{superheight}\opn\lcm{lcm}
\opn\trdeg{tr\,deg}%
\opn\reg{reg}
\opn\lreg{lreg}
\opn\set{set}
\opn\supp{Supp}
\opn\shad{Shad}
\opn\div{div}
\opn\Div{Div}
\opn\cl{cl}
\opn\Cl{Cl}
\opn\Spec{Spec}
\opn\Supp{Supp}
\opn\supp{supp}
\opn\Sing{Sing}
\opn\Ass{Ass}
\opn\Ann{Ann}
\opn\Rad{Rad}
\opn\Soc{Soc}
\opn\Ker{Ker}
\opn\Coker{Coker}
\opn\Im{Im}
\opn\Hom{Hom}
\opn\Tor{Tor}
\opn\Ext{Ext}
\opn\End{End}
\opn\Aut{Aut}
\opn\id{id}

\opn\nat{nat}
\opn\GL{GL}
\opn\SL{SL}
\opn\mod{mod}
\opn\ord{ord}
\opn\aff{aff}
\opn\con{conv}
\opn\relint{relint}
\opn\st{st}
\opn\lk{lk}
\opn\cn{cn}
\opn\core{core}
\opn\vol{vol}
\opn\gr{gr}

\def\pot#1#2{#1[\kern-0.28ex[#2]\kern-0.28ex]}

\opn\dirlim{\underrightarrow{\lim}}
\opn\invlim{\underleftarrow{\lim}}
\def\pnt{{\raise0.5mm\hbox{\large\bf.}}}

\def\Implies{\ifmmode\Longrightarrow \else
     \unskip${}\Longrightarrow{}$\ignorespaces\fi}
\def\implies{\ifmmode\Rightarrow \else
     \unskip${}\Rightarrow{}$\ignorespaces\fi}
\def\iff{\ifmmode\Longleftrightarrow \else
     \unskip${}\Longleftrightarrow{}$\ignorespaces\fi}

\let\:=\colon
\newtheorem{Theorem}{Theorem}[section]
\newtheorem{Lemma}[Theorem]{Lemma}
\newtheorem{Corollary}[Theorem]{Corollary}

\newtheorem{Remark}[Theorem]{Remark}

\newtheorem{Example}[Theorem]{Example}

\newtheorem{Definition}[Theorem]{Definition}

\newcommand{\beq}[1]{\begin{equation}\label{#1}}
\newcommand{\eeq}{\end{equation}}

\newtheorem{theorem}{Theorem}[section]
\newcommand{\bet}[1]{\begin{theorem}\label{#1}}
\newcommand{\eet}{\end{theorem}}
\newtheorem{lemma}[theorem]{Lemma}
\newcommand{\bel}[1]{\begin{lemma}\label{#1}}
\newcommand{\eel}{\end{lemma}}
\newcommand{\bep}{\begin{proof}}
\newcommand{\eep}{\end{proof}}
\let\epsilon=\varepsilon
\let\phi=\varphi
\let\kappa=\varkappa
\numberwithin{equation}{section}

\title{Stanley depth of Edge ideals}
\author[Muhammad Ishaq]{Muhammad Ishaq}
\address{Muhammad Ishaq, Abdus Salam School of Mathematical Sciences, GC
University, Lahore, 68-B New Muslim town Lahore, Pakistan}
\email{ishaq$\_\,$maths@yahoo.com}

\author[Muhammad Imran Qureshi]{Muhammad Imran Qureshi}
\address{Muhammad Imran Qureshi, Abdus Salam School of Mathematical Sciences, GC
University, Lahore, 68-B New Muslim town Lahore, Pakistan}
\email{imranqureshi18@gmail.com}

\thanks{The authors would like to express their gratitude to ASSMS of GC University Lahore for creating a very appropriate atmosphere for research work. This research is partially supported by HEC Pakistan.}

\begin{document}

\maketitle

\begin{abstract}
We give an upper bound for the Stanley depth of the edge ideal $I$ of a $k$-partite complete graph and show that Stanley's conjecture holds for $I$. Also we give an upper bound for the Stanley depth of the edge ideal of a $k$-uniform complete bipartite hypergraph.
\vskip 0.4 true cm
\noindent
\end{abstract}
\section{Introduction}
Let $S=K[x_1,\ldots,x_n]$ be the polynomial ring  in $n$ variables
over a field  $K$ and $M$ be a finitely generated $\ZZ^n$-graded
$S$-module. If $u\in M$ is a homogeneous element in $M$ and
$Z\subset \{x_1,\ldots,x_n\}$ then let $uK[Z]\subset M$ denote
the linear $K$-subspace of all elements of the form $uf$,
$f\in K[Z]$. This space is called a Stanley space of dimension
$|Z|$ if $uK[Z]$ is a free $K[Z]$-module. A Stanley decomposition
of module $M$ is a presentation of the $K$-vector space $M$ as
a finite direct sum of Stanley spaces
$$\mathcal{D}:M=\bigoplus^r _{i=1} u_i K[Z_i].$$ The number
$$\sdepth(\mathcal{D})=\min\{|Z_i|: i=1,\ldots,r\}.$$is called the Stanley depth of decomposition $\mathcal{D}$ and the number
\[\sdepth(M):=\max\{\sdepth(\mathcal{D}):\mathcal{D} \text{ is a
 Stanley decomposition of}\ M\}
\]
is called the Stanley depth of $M$. This is a combinatorial invariant
which does not depend on the characteristic of $K$.
In 1982 Stanley conjectured (see \cite{RPS}) that
$\sdepth M\geq\depth M$.
This conjecture has been proved in several special cases (for example see~\cite{A}, \cite{AP}, \cite{MI1},\cite{DP1}, \cite{DP} and \cite{PQ}) but it is still open in general.  A method to compute the Stanley depth is given in
\cite{HVZ}. Even when it does not provide the value of the
Stanley depth, this method allows one to obtain fairly good
estimations for the invariant of interest.\\
\indent The aim of this paper is to bound the Stanley depth of the edge ideal of a complete $k$-partite graph and an $s$-uniform complete bipartite hypergraph(see Lemma \ref{kpartite}, Theorem \ref{graphsdepth}). In Corollary \ref{sconjecture} we showed that Stanley's conjecture holds for the edge ideal of a complete $k$-partite graph.
\\ \textbf{Acknowledgement:} Both authors are grateful to Professor D. Popescu
for helpful discussions
during the preparation of this paper.
\section{Stanley depth of edge ideal of $k$-partite graph}
\begin{Definition}{\em
\noindent Let $G(V,E)$ be a graph with vertex set $V$ and edge set $E$. Then $G(V,E)$ is called a complete graph if every $e\subset V$ such that $|e|=2$ belongs to $E$.
}
\end{Definition}
\begin{Definition}{\em
A graph $G(V,E)$ with vertex set $V$ and edge set $E$ is called complete $k$-partite if the vertex set $V$ is partitioned into $k$ disjoint subset $V_1,V_2,\dots,V_k$ and $E=\{\{u,v\}: u\in V_i, v\in V_j, i\neq j\}$.}
\end{Definition}
\begin{Definition}{\em
Let $G$ be a graph. Then the edge ideal $I$ associated to $G$ is the squarefree monomial ideal $I=(x_i x_j: \{v_i,v_j\}\in E)$ of $S$.}
\end{Definition}
Now let $G$ be a complete $k$-partite graph with vertex set $V(G)=V_1\cup V_2\cup\dots\cup V_k$ with $|V_i|=r_i$, where $r_i\in \mathbb{N}$ and $2\leq r_1\leq \dots\leq r_k$.
Let $r_1+\dots+r_k=n$. Let $I_1=(x_1,\dots,x_{r_1})$,$I_2=(x_{r_1+1},\dots,x_{r_1+r_2}),\dots, I_k=(x_{r_1+\dots+r_{k-1}+1},\dots,x_n)$ be the monomial ideals in $S$.
Then the edge ideal of $G$ is of the form $$I=(\sum_{i\neq j} I_i\cap I_j).$$
We recall the method of Herzog, Vladoiu and Zheng \cite{HVZ} for computing the Stanley depth of a squarefree monomial ideal $I$ using posets. Let $G(I)=\{u_1,\dots,v_l\}$ be the set of minimal monomial generators of $I$. The characteristic poset of $I$ with respect to $h=(1,1,\dots,1)$(see \cite{HVZ}), denoted by $\mathcal{P}_I ^h$ is in fact the set
\[\mathcal{P}_I ^h=\{C\subset[n]\mid C \text{ contains the } \supp(u_i) \text{ for some } i\}\]
where $\supp(u_i)=\{j:x_j\mid u_i\}\subseteq[n]:=\{1,\dots,n\}$. For every $A,B\subset \mathcal{P}_I ^h$ with $A\subset B$, define the interval $[A,B]$ to be $\{C\in \mathcal{P}_I ^h: A\subseteq C\subseteq B\}$. Let $\mathcal{P}:\mathcal{P}_I ^h=\cup_{i=1}^r[C_i,D_i]$ be a partition of $\mathcal{P}_I ^h$, and for each $i$, let $c(i)\in \{0,1\}^n$ be the $n$-tuple such that $\supp(x^{c(i)})=C_i$. Then there is a Stanley decomposition $\mathcal{D(P)}$ of $I$
\[\mathcal{D(P)}: I=\bigoplus\limits_{i=1}^s x^{c(i)}K[\{x_k\mid k\in D_i\}].\]
By \cite{HVZ} we get that
\[\sdepth(I)=\max\{\sdepth(\mathcal{D(P)})\mid \mathcal{P} \text{ is a partition of } \mathcal{P}_I ^h\}.\]
\begin{Lemma}\label{kpartite}
$$\sdepth(I)\leq 2+\frac{\binom{n}{3}-(\sum\limits_{i=1} ^k\binom{r_i}{3})}{\sum\limits_{1\leq i< j\leq k} r_i r_j}$$
\end{Lemma}
\begin{proof}
Note that $I$ is a square free monomial ideal generated by monomials of degree 2. Let $d=\sdepth(I)$. The poset
$P_{I}^h$ has the partition
$\mathcal{P}: P_{I}^h=\bigcup_{i=1}^s[C_i,D_i]$, satisfying
$\sdepth(\mathcal{D(\mathcal{P})})=d$, where
$\mathcal{D(\mathcal{P})}$ is the Stanley decomposition of
$I$ with respect to the partition $\mathcal{P}$. We may
 choose $\mathcal{P}$ such that $|D|=d$ whenever $C\neq D$
in the interval $[C,D]$. Now we see that for each interval $[C,D]$ in $\mathcal{P}$ with $\mid C\mid=2$ we have $d-2$ subsets of cardinality 3 in this interval. The total number of these kind of intervals is $\sum\limits_{1\leq i< j\leq k} r_i r_j$ so we have $$(\sum\limits_{1\leq i< j\leq k} r_i r_j)(d-2)$$
subsets of cardinality 3. This number is less than or equal to the total number of subsets of cardinality 3 in $I$. So
$$(\sum_{1\leq i< j\leq k} r_i r_j)(d-2)\leq \binom {n}{3}-\binom{r_1}{3}-\dots-\binom{r_k}{3}$$
This implies
$$d\leq 2+\frac{\binom{n}{3}-\binom{r_1}{3}-\dots-\binom{r_k}{3}}{\sum\limits_{1\leq i< j\leq k} r_i r_j}$$
\end{proof}
\begin{Example}\label{Ex1}
{\em Let us consider $I=(I_1\cap I_2,I_1\cap I_3,I_1\cap I_4,I_2\cap I_3,I_2\cap I_4,I_3\cap I_4)$ be a monomial ideal in $S=K[x_1,\dots,x_{30}]$, where $I_1=(x_1,\dots,x_7),I_2=(x_8,\dots,x_{14}),I_3=(x_{15},\dots,x_{21}),I_4=(x_{22},\dots,x_{30})$.
\\Applying Lemma \ref{kpartite} we get $\sdepth(I)\leq 13$.}
\end{Example}
\begin{Lemma}\label{primarydec}
$\Ass(S/I)=\{P_1,\dots,P_k\}$ where \[P_i=({x_j\mid x_j \not\in I_i}), \forall\; i=1,\dots,k.\]
\end{Lemma}
\begin{proof}
We proceed as follows, Let $I=(I:x_1)\cap(I,x_1)$. We see that $(I:x_1)=P_1$. Let $J_{11}=(I,x_1)$.
Now $J_{11}=(J_{11}:x_2)\cap(J_{11},x_2)$ we have $(J_{11}:x_2)=(P_1,x_1)$
\\But we can omit $(J_{11}:x_2)$ because $P_1$ already appear in the primary decomposition. Proceeding in this way up to step $r_1$ we get $I=P_1\cap (I_1,\sum\limits_{2\leq i\neq j} I_i\cap I_j )$
\\Let $J_2=(I_1,\sum\limits_{2\leq i\neq j} I_i\cap I_j )$. Now we take $J_2=(J_2:x_{r_1+1})\cap(J_2,x_{r_1+1})$ and
we get $(J_2:x_{r_1+1})=P_2$. In this way, after $r_1+r_2$ steps we get $I=P_1\cap P_2\cap (I_1,I_2,\sum\limits_{3\leq i\neq j} I_i\cap I_j)$ and finally
 $I=P_1\cap\dots\cap P_k$.
\end{proof}
\begin{Definition}{\em
We call the big size of $I$ (see \cite{DP1}) the minimal number $t = t(I)<s$ such that the sum of all possible $(t + 1)$-prime
ideals of $\Ass(S/I)=\{P_1, . . . , P_k\}$ is the maximal ideal $(x_1,\dots,x_n)$.}
\end{Definition}
\begin{Corollary}\label{sconjecture}
Let $I$ be the edge ideal of complete $k$-partite graph then Stanley's conjecture holds for $I$.
\end{Corollary}
\begin{proof}
We see that the big size of $I$ is 1 by Lemma \ref{primarydec} so by \cite[Corollary 1.6]{DP1} (see also \cite[Theorem 1.2]{HPV}) Stanley's conjecture holds.
\end{proof}
Let $I'=(I,x_{n+1},\dots,x_{n+p})$ be a monomial ideal in $S'=S[x_{n+1},\dots,x_{n+p}]$.
Let denote by $A$ the upper bound of $\sdepth(I)$ found by Lemma \ref{kpartite}.
\begin{Theorem}\label{kpartite1}
Then \[\sdepth(I')\leq
2+\frac{\binom{n}{3}-\sum\limits_{i=1}^k\binom{r_i}{3}+\binom{p}{3}+n\binom{p}{2}+p\binom{n}{2}}{\sum\limits_{1\leq i<j\leq k} r_i r_j +np+\binom{p}{2}-\frac{p(A+p-1)}{2}}\]
where $\binom{a}{b}=0$ if $a<b$
\end{Theorem}
\begin{proof}
Note that $I'$ is a squarefree monomial ideal generated by
monomials of degree 2 and 1. Let $d=\sdepth(I')$. The poset
$P_{I'}$ has the partition
$\mathcal{P}: P_{I'}=\bigcup_{i=1}^s[C_i,D_i]$, satisfying
$\sdepth(\mathcal{D(\mathcal{P})})=d$, where
$\mathcal{D(\mathcal{P})}$ is the Stanley decomposition of
$I'$ with respect to the partition $\mathcal{P}$. We may
 choose $\mathcal{P}$ such that $|D|=d$ whenever $C\neq D$
in the interval $[C,D]$.

For each interval $[C_i,D_i]$ in  $\mathcal{P}$ with
$|C_i|=2$ when in the corresponding monomial the variables belong to $\{x_1,\ldots,x_n\}$ we have $|D_i|-|C_i|$ subsets of
cardinality 3 in this interval. We have $\sum\limits_{1\leq i<j\leq k} r_i r_j$ such intervals. Now for each interval
$[C_j,D_j]$ when $\mid C_j\mid=1$ we have at least $\binom{d-1}{2}$ subsets of cardinality 3 in this interval. We have $p$
such intervals. So we have
$p\binom{d-1}{2}$ subsets of cardinality 3.

 Now we consider those intervals
$[C_l,D_l]$ such that $|C_l|=2$ and the corresponding monomial
is of the form $x_l x_\lambda$, where
$x_l \in \{x_{n+1},\ldots,x_{n+p}\}$. Now either
$x_\lambda\in \{x_1,\ldots,x_n\}$ or
$x_\lambda\in \{x_{n+1},\ldots,x_{n+p}\}$. If
$x_\lambda\in \{x_1,\ldots,x_n\}$ then we have $np$
such intervals and each has at least $d-2$ subsets of
cardinality 3. If $x_\lambda\in\{x_{n+1},\ldots,x_{n+p}\}$
then we have
$\binom{p}{2}$ such intervals and each has at least $d-2$ subsets of
cardinality 3. Some subsets of cardinality 2 of the form
$C_l$ already appear in the intervals $[C_j,D_j]$ and such
subsets are $p(d-1)$ in number. Since the partition is
disjoint, we subtract this from total number of $C_l$'s,
so that we have at least
\[(\sum\limits_{1\leq i<j\leq k} r_i r_j)(d-2)+p\binom{d-1}{2}+\big[np+\binom{p}{2}-p(d-1)\big](d-2)\]
subset of cardinality 3. This number is less than or equal to the total number of subsets of cardinality 3. So
\[(\sum\limits_{1\leq i<j\leq k} r_i r_j)(d-2)+p\binom{d-1}{2}+\big[np+\binom{p}{2}-p(d-1)\big](d-2)\]
\[\leq \binom{n}{3}-\sum\limits_{1\leq i<j\leq k} \binom{r_i}{3}+\binom{p}{3}+n\binom{p}{2}+p\binom{n}{2}\]
Now we know by Lemma \ref{kpartite} and \cite[Lemma 2.11]{I} that
$d\leq A+p$. This implies
$-(d-1)\geq -A-p+1$. As in \cite{CQ}, using this in the left side
of above inequality, one gets
\[\big(\sum\limits_{1\leq i<j\leq k} r_i r_j +np+\binom{p}{2}-\frac{p(A+p-1)}{2}\big)(d-2)\]
\[\leq(\sum\limits_{1\leq i<j\leq k} r_i r_j)(d-2)+p\binom{d-1}{2}+\big[np+\binom{p}{2}-p(d-1)\big](d-2)\]
Combining both inequalities we get the required result.
\end{proof}
\begin{Example}{\em
Let $I'=(I,x_{31},\dots,x_{40})\subset S'=S[x_{31},\dots,x_{40}]$ be a monomial ideal, where $I$ is the same ideal as in Example \ref{Ex1}. Then by \cite[Theorem 2.11]{I} $\sdepth(I')\leq 23$. We see that $n=30,k=4,p=10,A=13,r_1 =7,r_2 =7,r_3 =7,r_4 =9$ Now by our Theorem \ref{kpartite1} we have $\sdepth(I')\leq 18$.}
\end{Example}
\bigskip
\section{Stanley depth of edge ideal of an $s$-uniform complete bipartite hypergraph}
 \begin{Definition}{\em
Let $V=\{v_1,\dots,v_t\}$ be a finite set and $E=\{E_1,\dots,E_r\}$ be a collection of distinct subsets of $V$. The pair $\mathcal{G}=\mathcal{G}(V,E)$ is said to be a hypergraph if $E_j\neq\emptyset$ for all $j$, where $V$ and $E$ are the sets of vertices and edges of $\mathcal{G}$ respectively.}
\end{Definition}
A hypergraph is said to be $s$-uniform hypergraph, if $|E_j|=s$ for all $j$.
\begin{Definition}{\em
Associate to each vertex $v_j$ of a hypergraph $\mathcal{G}$ a variable $x_j$ of a polynomial ring $S=K[x_1,\dots,x_t]$ then the edge ideal of $\mathcal{G}$ is defined as \[E(\mathcal{G})=\big(\{\prod\limits_{v_i\in E_j} x_i: E_j\in E\}\big )\subset S\]}
\end{Definition}
Let $I\subset S$ be the edge ideal of a complete bipartite graph over $n$ vertices with $n\geq 4$ then by Ishaq \cite[Theorem 2.8]{I}, we have $$\sdepth(I)\leq \frac{n+2}{2}.$$
Now our aim is to give an upper bound for the Stanley depth of an edge ideal of a hypergraph which is a kind of generalization to the complete bipartite graph. \\
\indent We say that $\mathcal{G}_s(V,E)$ is an $s$-uniform complete bipartite hypergraph if the following conditions holds
\begin{enumerate}
\item The vertex set $V$ is partitioned into $2$ disjoint subsets $V_1$ and $V_2$.
\item For all hyperedges $E_i$, $E_i\cap V_j\neq E_i$, $j=1,2$.
\item Each $s$-subset of $V$ such that $F\cap V_j\neq F$ for $j=1,2$ belongs to $E$.
\end{enumerate}
If $s=2$ then the hypergraph $\mathcal{G}_2(V,E)$ is just a complete bipartite graph.
\begin{Example}
Let $\mathcal{G}_3(V_1\cup V_2,E)$ be a 3-uniform bipartite hypergraph with $|V_1|=3$ and $|V_2|=3$. Then the edge ideal of $\mathcal{G}_3(V_1\cup V_2,E)$ is\\ $I=(x_1x_2x_4,x_1x_2x_5,x_1x_2x_6,x_1x_3x_4,x_1x_3x_5,x_1x_3x_6,x_2x_3x_4,x_2x_3x_5,x_2x_3x_6,x_1x_4x_5,\\x_2x_4x_5,
x_3x_4x_5,x_1x_4x_6,x_2x_4x_6,x_3x_4x_6,x_1x_5x_6,x_2x_5x_6,x_3x_5x_6)\subset K[x_1,\dots,x_6]$.
\end{Example}
Let $I_s\subset K[x_1,x_2,\dots,x_{|V|}]$ denote the monomial edge ideal of the hypergraph $\mathcal{G}_s(V,E)$. Then
\begin{Theorem}\label{graphsdepth}
\[s\leq \sdepth(I_s)\leq s+\frac{\binom{|V|}{s+1}-\binom{|V_1|}{s+1}-\binom{|V_2|}{s+1}}{\binom{|V|}{s}-\binom{|V_1|}{s}-\binom{|V_2|}{s}}.\]

\end{Theorem}
\begin{proof}
Note that $I_s$ is a squarefree monomial ideal generated by squarefree monomials of degree $s$. By \cite[Lemma 2.1]{S1} $s\leq \sdepth(I_s)$. Now we count the number of monomials of degree $s$ in $I_s$. To count the number of monomials of degree $s$ in $I$ we have to count the number of hyperedges of cardinality $s$ in $\mathcal{G}_s(V,E)$. The hypergraph $\mathcal{G}_s(V,E)$ contains all the edges $E_i$ of cardinality $s$ such that $E_i\cap V_j\neq E_i$, $j=1,2$. This means that a hyperedge $C$ of cardinality $s$ does not belongs to $\mathcal{G}_s(V,E)$ if $C\subset V_j$ for some $j$. Now let $N$ denotes the number of hyperedges which belongs to $\mathcal{G}_s(V,E)$. Then $N=\binom{|V|}{s}-\binom{|V_1|}{s}-\binom{|V_2|}{s}$ and if $s>|V_j|$ for some $j$ then we take $\binom{|V_j|}{s}=0$. Similarly to count the number of squarefree monomials of degree $s+1$ in $I$, we have to count the number of hyperedges of the hypergraph $\mathcal{G}_{s+1}(V,E)$. Let $M$ be the number of hyperedges of the hypergraph $\mathcal{G}_{s+1}(V,E)$ then as before $M={\binom{|V|}{s+1}-\binom{|V_1|}{s+1}-\binom{|V_2|}{s+1}}$ and if $s+1>|V_j|$ for some $j$ then we take $\binom{|V_j|}{s+1}=0$. By repeating the proof of Lemma \ref{kpartite} for $I_s$ we have $\sdepth(I_s)\leq s+\frac{M}{N}$ and the required result follows.
\end{proof}
\begin{Example}{\em
Let $I\subset S=K[x_1,\ldots,x_{15}]$ be the edge ideal of the hypergraph $\mathcal{G}_5(V,E)$ with $|V|=15$, $|V_1|=7$ and $|V_2|=8$. Then by Theorem \ref{graphsdepth} we have $5\leq \sdepth(I)\leq 6$.}
\end{Example}
\begin{Remark}{\em
Let $I$ be the edge ideal of the hypergraph $\mathcal{G}_s(V,E)$, if $s>|V_1|,|V_2|$ then $I$ is the squarefree Veronese ideal $I_{|V|,\,s}$.}
\end{Remark}

\end{document}